\documentclass[12pt,letterpaper,reqno]{amsart}

\usepackage{comment}
\usepackage{graphicx} % Required for inserting images
\usepackage{amsfonts}
\usepackage{amsmath}
\usepackage{amsthm}
\usepackage{xcolor}
\usepackage{mathtools}
\usepackage{fancyhdr,amssymb}
\usepackage{url}
\usepackage{enumerate}
\usepackage{float}
\allowdisplaybreaks

\newcommand{\hit}{\textit{higher-index terms}}
\newcommand{\lit}{\textit{lower-index terms}}

\def\<{\langle}
\def\>{\rangle}

\newtheorem{theorem}{Theorem}[section]
\newtheorem{proposition}[theorem]{Proposition}
\newtheorem{corollary}[theorem]{Corollary}
\newtheorem{definition}[theorem]{Definition}

\newtheorem{remark}[theorem]{Remark}

\newtheorem{conjecture}[theorem]{Conjecture}
\newtheorem{lemma}[theorem]{Lemma}

\newtheorem{observation}[theorem]{Observation}

\usepackage{tikz}
\usetikzlibrary{arrows.meta}
\usetikzlibrary{positioning}

\begin{document}

\pagestyle{fancy}
\fancyhf{} % Clear all headers and footers
\lhead{} % Left header - empty
\rhead{} % Right header - empty
\cfoot{} % Clear the footer (leave empty if you don't want page numbers)
% Remove header and footer lines
\renewcommand{\headrulewidth}{0pt} % No line in header
\renewcommand{\footrulewidth}{0pt} % No line in footer

%%%%%%%%%%%%%%%%%%%%%%%%%%%%%%%%%%%%%%%%%

\title{The Ordered Zeckendorf Game}

\author{Ivan Bortnovskyi}\email{ib538@cam.ac.uk}
\address{Department of Pure Mathematics and Mathematical Statistics, Univeristy of Cambridge, Cambridge, United Kingdom, CB3 0WA}

\author{Michael Lucas}\email{ml2130@cam.ac.uk}
\address{Department of Pure Mathematics and Mathematical Statistics, Univeristy of Cambridge, Cambridge, United Kingdom, CB3 0WA}

\author{Steven J. Miller}\email{Steven.J.Miller@williams.edu}
\address{Department of Mathematics and Statistics, Williams College,
Williamstown, MA 01267}

\author{Iana Vranesko}\email{yv2@williams.edu}
\address{Department of Mathematics, Williams College, Williamstown, MA, 01267}

\author{Ren Watson}\email{renwatson@utexas.edu}
\address{Department of Mathematics, University of Texas at Austin, Austin, TX 78703}

\author{Cameron White}\email{cjw3@williams.edu}
\address{Department of Mathematics, Williams College, Williamstown, MA, 01267}

\date{\today}


\thanks{The first listed author was supported by The Winston Churchill Foundation of the United States. The NSF grant DMS2242623 supported the fifth listed author, and the fourth and sixth authors were supported by Williams College and the Finnerty fund. Finally, the second listed author was supported by the Dr. Herchel Smith Fellowship Fund.}

\maketitle

\begin{abstract}
We introduce and analyze the \emph{ordered Zeckendorf game}, a novel combinatorial two-player game inspired by Zeckendorf’s Theorem, which guarantees a unique decomposition of every positive integer as a sum of non-consecutive Fibonacci numbers. Building on the original Zeckendorf game—previously studied in the context of unordered multisets—we impose a new constraint: all moves must respect the order of summands. The result is a richer and more nuanced strategic landscape that significantly alters game dynamics.

Unlike the classical version, where Player 2 has a dominant strategy for all $n > 2$, our ordered variant reveals a more balanced and unpredictable structure. In particular, we find that Player 1 wins for nearly all values $n \leq 25$, with a single exception at $n = 18$. This shift in strategic outcomes is driven by our game's key features: adjacency constraints that limit allowable merges and splits to neighboring terms, and the introduction of a switching move that reorders pairs.

We prove that the game always terminates in the Zeckendorf decomposition—now in ascending order—by constructing a strictly decreasing monovariant. We further establish bounds on game complexity: the shortest possible game has length exactly $n - Z(n)$, where $Z(n)$ is the number of summands in the Zeckendorf decomposition of $n$, while the longest game exhibits quadratic growth, with $M(n) \sim \frac{n^2}{2}$ as $n \to \infty$.

Empirical simulations suggest that random game trajectories exhibit log-normal convergence in their move distributions. Overall, the ordered Zeckendorf game enriches the landscape of number-theoretic games, posing new algorithmic challenges and offering fertile ground for future exploration into strategic complexity, probabilistic behavior, and generalizations to other recurrence relations.
\end{abstract}

\tableofcontents

%%%%%%%%%%%%%%%%%%%%%%%%%%%%%%%%%%%%%%%%%%%%%%%%%%%%%%%%%%
%%%%%%%%%%%%%%%%%%%%%%%%%%%%%%%%%%%%%%%%%%%%%%%%%%%%%%%%%%

\section{Introduction}

\subsection{History and Motivation}

The Fibonacci sequence is given by $F_1=1,F_2=2$, and $F_{i+1}=F_{i}+F_{i-1}$ for $i \geq 2$. Zeckendorf’s Theorem \cite{Zeckendorf1972} states that every positive integer can be uniquely represented as a sum of non-consecutive Fibonacci numbers, noting that in order to guarantee uniqueness it is necessary that we have only one term equal to $1$. This decomposition can be efficiently obtained via a greedy algorithm, and its structural and probabilistic properties have been extensively investigated across combinatorics, number theory, and discrete mathematics. A game-theoretic interpretation of the theorem was first proposed in~\cite{ZeckGameOriginal}, where two players alternately perform legal replacements on a multiset of Fibonacci numbers until reaching the unique Zeckendorf decomposition. A formal definition of the classical Zeckendorf Game is given below.

\begin{definition}[The Two Player Zeckendorf Game]
    At the beginning of the game, there is an unordered list of $n$ 1's: $F_1 \land F_1 \land \cdots \land F_1$. Players alternate turns, and on each turn, a player may make one of the following moves.
    \begin{enumerate}
        \item \textbf{Merging:} Replace $F_i \land F_{i+1}$ with $F_{i+2}$.
        \item \textbf{Merging Ones:} Replace $F_1\land F_1$ with $F_2$.
        \item \textbf{Splitting:} Replace $F_i \land F_i$ with $F_{i+1} \land F_{i-2}$ for $i >2$.
        \item \textbf{Splitting Twos:} Replace $F_2 \land F_2$ with $F_3 \land F_1$.
    \end{enumerate} The game ends when no further moves can be made.
\end{definition}

In the classical Zeckendorf game, the order of summands is irrelevant—players operate on unordered collections of Fibonacci numbers, applying merge or split operations regardless of position. It is known that in this formulation, Player 2 has a winning strategy for all $n > 2$, which significantly limits the strategic complexity of the game~\cite{BairdSmith2019}.

To address these limitations and explore deeper combinatorial structures, we introduce a new variant: the \emph{ordered Zeckendorf game}. In our version, the game begins with an ordered $n$-tuple of $F_1 = 1$, and all allowed moves are restricted to adjacent elements. The resulting dynamics differ significantly from those in the unordered setting, as adjacency constraints restrict player choices and emphasize positional strategy. Remarkably, we find that for all $n \leq 25$—with the sole exception of $n = 18$, where Player 2 has a winning strategy—Player 1 can force a win. This is in sharp contrast to the classical game and suggests that the introduction of order restores balance and increases strategic depth.

Our variant introduces a novel move: the \emph{switch}, which allows players to swap adjacent elements out of order. This rule further enhances the game's complexity, producing larger game trees and requiring careful local planning. The interaction between merging, splitting, and switching operations creates a rich environment for strategic exploration, raising new algorithmic and theoretical questions.

This paper presents a comprehensive analysis of the ordered Zeckendorf game, combining theoretical results with computational observations. We prove that every game terminates in the ascending Zeckendorf decomposition of $n$, using a carefully constructed monovariant. We also establish lower and upper bounds on game length and propose a conjectural strategy for maximizing the number of moves. Finally, we investigate the distribution of random play lengths and observe evidence suggesting convergence to a log-normal distribution.

Our work contributes to an expanding literature on Zeckendorf-type games and generalized decompositions. For readers seeking additional context or further directions, we recommend the following references:
\begin{itemize}
    \item Beckwith et al.~\cite{Beckwith2013}, on average gap distributions in generalized Zeckendorf decompositions,
    \item Kopp et al.~\cite{Kopp2011}, on the number of summands in Zeckendorf decompositions,
    \item Gueganic et al.~\cite{Gueganic2019}, on limiting distributions in generalized settings,
    \item Baird-Smith et al.~\cite{BairdSmith2019}, on the generalized Zeckendorf game,
    \item Boldyriew et al.~\cite{Boldyriew2020}, on extensions to non-constant recurrence relations,
    \item Cusenza et al.~\cite{Cusenza2021}, on multiplayer and alliance-based Zeckendorf games.
\end{itemize}

\subsection{Main Results}

We use \(F_i\) to denote the \(i^{\text{th}}\) Fibonacci number, where \(F_1 = 1\), \(F_2 = 2\), and the recurrence relation \(F_{i+1} = F_i + F_{i-1}\) holds for all \(i \geq 2\).

\begin{definition}[The Two-Player Ordered Zeckendorf Game]
At the start of the game, the state is an ordered list of \(n\) copies of \(F_1\): \((F_1, F_1, \dots, F_1)\). Players alternate turns, and on each turn, a player may perform one of the following five legal moves.
\begin{enumerate}
    \item \textbf{Merging:} Replace a pair of adjacent terms \((F_i, F_{i+1})\) with \(F_{i+2}\).
    \item \textbf{Merging Ones:} Replace \((F_1, F_1)\) with \(F_2\).
    \item \textbf{Splitting:} Replace a pair \((F_i, F_i)\) with \((F_{i-2}, F_{i+1})\), for \(i > 2\).
    \item \textbf{Splitting Twos:} Replace \((F_2, F_2)\) with \((F_1, F_3)\).
    \item \textbf{Switching:} Swap adjacent elements \((F_i, F_j) \rightarrow (F_j, F_i)\) whenever \(i > j\).
\end{enumerate}
The game ends when no further moves can be applied; that is, when the list of Fibonacci numbers is strictly increasing and no merge or split move is possible.
\end{definition}

These sum-preserving operations derive naturally from the Fibonacci recurrence and the constraints imposed by ordering. Merges and splits allow construction and deconstruction of Fibonacci terms, while switches restore ascending order to enable further operations. Together, they encode a localized, combinatorial decision-making framework that is both computationally rich and strategically deep.

Our main theoretical results fall into three broad categories: termination and correctness of the game, bounds on its length, and structural conjectures about optimal play. 

\begin{theorem}[Termination and Final State]
\label{thm:termination}
The ordered Zeckendorf game always terminates in a finite number of moves. Moreover, the final state is the Zeckendorf decomposition of \(n\), expressed as a strictly increasing sequence of Fibonacci numbers.
\end{theorem}

The proof of this result proceeds by defining a suitable monovariant that decreases strictly with each move. The monovariant simultaneously tracks disorder (via inversions) and total weight (via the sum of Fibonacci indices), ensuring that the game cannot cycle and must converge. Importantly, the final configuration must be the unique Zeckendorf decomposition, as all alternative representations are eliminated by the move set.
Since the game always terminates in a unique configuration, we can define natural extremal questions: What is the minimal number of moves required? What is the maximum? We begin with the lower bound.

\begin{proposition}[Shortest Game]
\label{Shortest Game}
The shortest possible game has length \(n - Z(n)\), where \(Z(n)\) is the number of terms in the Zeckendorf decomposition of \(n\).
\end{proposition}

This minimal length corresponds to a greedy strategy in which players always combine adjacent pairs as quickly and effectivelly as possible to reach the final state with minimal intermediate steps.

At the other extreme, we are interested in how long the game can be prolonged under optimal play designed to delay termination.

\begin{theorem}[Upper Bound on the Maximal Game Length]\label{thm:upper-bound-termination}
Let \( M(n) \) denote the maximal length of the Ordered Zeckendorf game on $n$. Then
\[
M(n)~\leq~\frac{n(n-1)}{2}.
\]
\end{theorem}

To complement this upper bound, we provide a lower bound for $M(n)$ by proposing a specific strategy that aims to maximize the game length.

\begin{definition}[Long Game Strategy]
\label{def:long-strat}
    The Long Game Strategy (LGS) is defined as the following priority-based strategy.
    \begin{enumerate}
        \item Perform \textbf{switch} moves (in any order).
        \item Combine adjacent \textbf{ones}, starting from the left.
        \item Perform \textbf{split} moves, starting from the right.
        \item Perform \textbf{merge} moves, starting from the left.
\end{enumerate}
\end{definition}

\begin{conjecture}[Strategy for Maximum Length]
\label{Longest strategy}
The LGS has the longest game length.
\end{conjecture}

This conjecture is supported by exhaustive simulations, providing strong empirical evidence for its effectiveness. Building on this, we establish a lower bound on the maximal game length by following the LGS. This bound requires an auxiliary lemma about the number of repetitions of $F_i$ for $i \geq 3$ under the LGS. As we establish our lower bound for maximal game length by lower bounding the number of moves necessary to satisfy certain criteria on the number of copies of $F_1$ and $F_2$ present in the current decomposition, it is helpful to make the following distinction.

\begin{definition}[Higher-Index Terms and Lower-Index Terms]
    We call a Fibonacci number $F_i$ a \emph{higher-index term} if $i \geq 3$, and a \emph{lower-index term} if $i \leq 2$.
\end{definition}

\begin{lemma}[Repetitions in Higher Index Terms]
\label{lem:higher-repetitions}
    Suppose, at some moment, after reordering and merging $F_1$'s, the tuple has a form of
    $$(\lit, \hit).$$ 
    Then the \text{higher-index terms} subtuple is either of the following.
    \begin{itemize}
        \item There are no repeating Fibonacci numbers.
        \item There is a single repetition, of one of the following forms:
        \begin{enumerate}
            \item $F_3, F_3, F_{i\geq 4}, \dots$, \label{state:33}
            \item $F_4, F_4, F_{i\geq 5}, \dots$, \label{state:44}
            \item $\dots, F_{j-x}, F_j, F_j, F_{i\geq j+1}, \dots$, where $x\geq 3$. \label{state:jj}
        \end{enumerate}
    \end{itemize}
    In particular, there is at most one repetition in the \hit. We display the possible ways to move between the states "No Repetitions", \ref{state:33}, \ref{state:44}, and \ref{state:jj} in a state-transition graph in Figure \ref{fig:higher-repetitions}.
\end{lemma}
\begin{figure}[H]
\begin{tikzpicture}[->, >=stealth, node distance=2.5cm, every node/.style={draw, minimum size=1cm}, align=center]

% Nodes
\node[circle] (1) {1};
\node[circle, right of=1] (2) {2};
\node[circle, right of=2, yshift=1cm] (3) {3};
\node[rectangle, above left of=1, xshift=-0.5cm, yshift=1cm, minimum width=3.5cm, minimum height=1.5cm] (N) {No Repetitions};

% Arrows between nodes
\draw (1) -- (2);
\draw (2) -- (3);
\draw (2) to[bend right] (N);
\draw (3) to[bend right] (N);

% Double arrow between 1 and No repetitions
\draw (1) to[bend left=20] (N);
\draw (N) to[bend left=20] (1);

% Self-loops
\draw (3) edge[loop right] ();
\draw (N) edge[loop above] ();

\end{tikzpicture}
\caption{Directed graph showing the possible ways to move between game states in Lemma \ref{lem:higher-repetitions}.}
\label{fig:higher-repetitions}
\end{figure}

\begin{theorem}[Lower Bound of Maximal Game Length]
\label{thm:lower-bound-termination}
The LGS takes at least
\[
\frac{n^2}{2} - n \log_\phi(n) + o(n \log_\phi(n))
\]
moves\footnote{Here $\phi=\frac{1+\sqrt{5}}{2}$ is the golden ratio.}.
\end{theorem}

 It turns out that we can refine slightly a bound obtained in Theorem \ref{thm:upper-bound-termination} by similar expression:

\begin{theorem}[Refined Upper Bound on the Maximal Game Length]
\label{thm:refined-upper-bound-termination}
    \[
    M(n) ~\leq~ \dfrac{n^2}{2} - \dfrac{1}{32}n\log_\phi n + o(n\log_\phi n).
    \]
\end{theorem}

Combining the upper and lower bounds, we arrive at the following asymptotic characterization.

\begin{corollary}[Asymptotic Behavior of Maximal Game Length]
\label{cor:bound-termination}
We have
\[
M(n)  ~=~ \frac{n^2}{2} - \Theta(n \log_\phi(n)).
\]
\end{corollary}

\section{Termination and Final State}

We begin by establishing an upper bound on maximal game length for the Ordered Zeckendorf Game, providing as a corollary game termination and final state.

\begin{proof} [Proof of Theorem \ref{thm:upper-bound-termination}]
We define a monovariant: for a state $S = (F_{i_1}, \dots, F_{i_k})$, let
\[
f(S) ~= ~\sum_{j=1}^k (k+1 - j) F_{i_j}.
\]
Note that $f(S) \geq \sum_{j=1}^k F_{i_j} = n$. \\

Now, let us show that each legal move reduces this quantity.
\begin{itemize}
    \item When merging $(F_i, F_{i+1})$ to $F_{i+2}$ in the $j^{\text{th}}$ position, the weights of all terms to the left of $(F_i, F_{i+1})$ in $S$ are decreased by 1. The change in the function at $(F_i, F_{i+1})$ is
    \begin{align*}
        & (k-j)F_{i+2} - (k+1-j)F_i - (k-j)F_{i+1} \\
        &=~ (k-j)(F_{i+2} - F_i - F_{i+1}) - F_i \\
        &=~ - F_i \\
        &<~ 0.
    \end{align*}
    \item When splitting $(F_i,F_i)$ to $(F_{i-2},F_{i+1})$, the weights on all other terms stay the same. The change in the function is therefore
    \begin{align*}
       & (k+1-j)F_{i-2} + (k-j)F_{i+1} - (k+1-j)F_i - (k-j)F_i \\
       &=~ (k-j)(F_{i-2} + F_{i+1} - 2F_i) + F_{i-2} - F_i \\
       &=~ - F_{i-1} \\
       &<~ 0.      
    \end{align*}
    \item When splitting $(F_2,F_2)$ to $(F_1,F_3)$, the weights on all other terms stay the same. The change in the function is therefore
    \begin{align*}
       & (k+1-j)F_1+ (k-j)F_3 - (k+1-j)F_2 - (k-j)F_2 \\
       &=~ (k-j)(F_1 + F_3 - 2F_2) + F_1 - F_2 \\
       &=~ -1.      
    \end{align*}
    \item When merging ones, the weights of all summands to the left are decreased by at least 1, and the value of the function at the pair of ones decreases by 1, so the function decreases.
    \item When switching $(F_{i_j}, F_{i_{j+1}})$, given that \(F_{i_j} > F_{i_{j+1}}\), the weights of all other summands stay the same, so the change in the function is
    \begin{align*}
        &(k+1-j)F_{i_{j+1}} + (k-j)F_{i_j} - (k+1-j)F_{i_j} - (k-j)F_{i_{j+1}} \\
        &=~ F_{i_{j+1}} - F_{i_j} \\
        &<~ 0.
    \end{align*}
\end{itemize}

Since $f$ begins at $n(n+1)/2$, decreases by at least 1 per move, and ends at at least $n$, the number of moves is bounded above by ${n \choose 2}$. The final configuration must be the ordered Zeckendorf decomposition because any configuration not satisfying the Zeckendorf condition allows further moves, which consequentually proves Theorem \ref{thm:termination}.
\end{proof}

\begin{proof}[Proof of Theorem \ref{thm:termination}]
    Proof of Theorem \ref{thm:termination} follows as an immediate corollary from the above proof of the Theorem \ref{thm:upper-bound-termination}.
\end{proof}

%%%%%%%%%%%%%%%%%%%%%%%%%%%%%%%%%%%%%%%%%%%%%%%%%%%%%%%
%%%%%%%%%%%%%%%%%%%%%%%%%%%%%%%%%%%%%%%%%%%%%%%%%%%%%%%

\section{Game Length: More Bounds and Extremes}

In order to establish a lower bound for maximal game length, we first prove our structural lemma verifying the existence of at most one repetition in the higher-index terms of a current decomposition in a LGS game. A state-transition graph accompanying this proof is shown in Figure \ref{fig:higher-repetitions}.

\begin{proof}[Proof of Lemma \ref{lem:higher-repetitions}]
    We proceed by an inductive argument to prove the transitions shown in Figure \ref{fig:higher-repetitions}. Namely, at the start of the game, \hit\ does not contain any elements, so it is in \textit{No Repetitions} state. We prove that regardless of the move \hit\ remains in one of four states described above.
    
    For the rest of the proof we suppose there are at least two consecutive $F_2$ terms. Suppose that after reordering and merging $F_1$'s we are in \textit{No Repetitions} state. Then, according to the priority of moves in the LGS algorithm, the next move will involve splitting  two rightmost $F_2$'s. 
    
    If there is no $F_3$ at the beginning of \hit, then we still remain in \textit{No Repetitions}, otherwise we transition into \textit{State 1}.
    
    In \textit{State 1}, we have to split two $F_3$'s into $F_1, F_4$, so that after reordering and merging of $F_1$'s \hit\ will be of the form:
    $$F_4, ~F_{i\geq 4}, ~\dots $$
    If $i>4$, then we get back to the \textit{No Repetitions}, otherwise we transition into \textit{State 2}. 
    
    In the \textit{State 2}, we have to split two $F_4$'s into $F_2, F_5$ resulting \hit\ to be
    $$F_5,~ F_{i\geq5}, ~\dots.$$
    If $i>5$ then we get to the \textit{No Repetitions} state, or otherwise into \textit{State 3}, as there are no $F_3$ or $F_4$ terms directly to the left of the leftmost $F_5$.

    In the \textit{State 3}, we have to split two $F_j$'s making \hit\ of the form
    $$\dots, ~F_{j-x}, ~F_{j-2}, ~F_{j+1}, ~F_{i\geq j+1}, ~\dots.$$
    Note that since $x\geq3$, terms $F_{j-x}$ and $F_{j-2}$ are different, so there will be no repetition in this pair. The only repetition can be in the pair $F_{j+1}, F_{i\geq j+1}$ if $i=j+1$. Then we remain in \textit{State 3}, otherwise we go back to \textit{No Repetitions}, concluding the proof.
\end{proof}

Utilizing the above result, we are now able to establish our lower bound for maximum game length.

\begin{proof}[Proof of Theorem \ref{thm:lower-bound-termination}]
We begin the game by merging the leftmost two copies of $F_1$, then moving the copy of $F_2$ created by this move to the rightmost end of the tuple via $n-2$ switch moves. We are left with the tuple $(F_1,F_1,\dots,F_1,F_2)$. As no more switch moves remain, we repeat this process: letting $k$ denote the number of copies of $F_1$ remaining, we merge the leftmost two copies of $F_1$ into a copy of $F_2$ in a single move and then perform $k-2$ switch moves to move the new copy of $F_2$ to the right of the $k-2$ remaining copies of $F_1$, yielding in $k-1$ moves a tuple consisting of $k-2$ copies of $F_1$ to the left of $\frac{n-k-2}{2}$ copies of $F_2$. We continue iterating this process until $k \leq 2$, noting that we begin with $k=n$ and decrement $k$ by $2$ at each iteration. Thus, this portion of the games terminates in $\sum_{j=0}^{\lfloor\frac{n-2}{2}\rfloor}(n-2j-1)$ moves. If $n$ is odd, this summation is equal to $(n-1)+(n-3)+\dots+1=\frac{n^2}{4}$ moves and this portion of the game terminates at the tuple $(F_1,F_2,F_2,\dots,F_2)$, and if $n$ is even, this summation is equal to $(n-1)+(n-3)+\dots+2=\frac{(n-1)^2}{4}+\frac{n-1}{2}=\frac{n^2-1}{4}$ and this portion of the game terminates at the tuple $(F_2,F_2,\dots,F_2)$.\\Now, we suppose the tuple has size $k$ and we write it as $(F_2, \dots, F_2, \hit)$, where \hit\ denotes higher-index terms that are at least $F_3$. We consider the number of moves taken to convert it into the $(k-1)$-tuple $(F_2, \dots, F_2, \hit)$. \textit{Note that the first $k$-tuple only contains a copy of $F_1$ when $n$ is odd and $k=\lceil\frac{n}{2}\rceil$. We ignore this case in our calculations as it only increases game length.} By Lemma \ref{lem:higher-repetitions}, there is at most one repeated higher-index term, so the number of higher-index terms is bounded by $\ell-1$, where $F_\ell$ is the maximal Fibonacci summand in the Zeckendorf decomposition of $n$. By \cite{BairdSmith2019}, $\ell \leq \log_\phi(\sqrt{5}n+1/2)$, so the number of higher-index terms is bounded by \begin{equation}\label{def:c_n} c_n := \log_\phi(\sqrt{5}n+1/2) -1.\end{equation}Hence, there are at least $k-c_n-1$ copies of $F_2$ in the $k$-tuple.\\We assume that there are at least four copies of $F_2$ and that there are no copies of $F_1$. This holds for $c_n+5 \leq k \leq \lfloor \frac{n}{2} \rfloor$.

\begin{figure}

\begin{tikzpicture}[x=1pt,y=1pt,yscale=-1,xscale=1]
%uncomment if require: \path (0,300); %set diagram left start at 0, and has height of 300

%Shape: Rectangle [id:dp35404179858921747] 
\draw  [fill={rgb, 255:red, 184; green, 233; blue, 134 }  ,fill opacity=1 ] (88.67,21.33) -- (150.33,21.33) -- (150.33,50.67) -- (88.67,50.67) -- cycle ;
%Shape: Circle [id:dp9109041995676722] 
\draw   (30.67,109.83) .. controls (30.67,99.06) and (39.4,90.33) .. (50.17,90.33) .. controls (60.94,90.33) and (69.67,99.06) .. (69.67,109.83) .. controls (69.67,120.6) and (60.94,129.33) .. (50.17,129.33) .. controls (39.4,129.33) and (30.67,120.6) .. (30.67,109.83) -- cycle ;
%Shape: Circle [id:dp0048301331881160925] 
\draw   (99.33,109.83) .. controls (99.33,99.06) and (108.06,90.33) .. (118.83,90.33) .. controls (129.6,90.33) and (138.33,99.06) .. (138.33,109.83) .. controls (138.33,120.6) and (129.6,129.33) .. (118.83,129.33) .. controls (108.06,129.33) and (99.33,120.6) .. (99.33,109.83) -- cycle ;
%Shape: Circle [id:dp09818135066238687] 
\draw  [color={rgb, 255:red, 208; green, 2; blue, 27 }  ,draw opacity=1 ][dash pattern={on 4.5pt off 4.5pt}] (168.67,110.5) .. controls (168.67,99.73) and (177.4,91) .. (188.17,91) .. controls (198.94,91) and (207.67,99.73) .. (207.67,110.5) .. controls (207.67,121.27) and (198.94,130) .. (188.17,130) .. controls (177.4,130) and (168.67,121.27) .. (168.67,110.5) -- cycle ;
%Shape: Rectangle [id:dp9976609585592056] 
\draw  [color={rgb, 255:red, 208; green, 2; blue, 27 }  ,draw opacity=1 ][dash pattern={on 4.5pt off 4.5pt}][line width=0.75]  (309.33,21.33) -- (371,21.33) -- (371,50.67) -- (309.33,50.67) -- cycle ;
%Shape: Circle [id:dp7484437025518068] 
\draw  [fill={rgb, 255:red, 184; green, 233; blue, 134 }  ,fill opacity=1 ] (251.33,109.83) .. controls (251.33,99.06) and (260.06,90.33) .. (270.83,90.33) .. controls (281.6,90.33) and (290.33,99.06) .. (290.33,109.83) .. controls (290.33,120.6) and (281.6,129.33) .. (270.83,129.33) .. controls (260.06,129.33) and (251.33,120.6) .. (251.33,109.83) -- cycle ;
%Shape: Circle [id:dp5175208727435856] 
\draw  [color={rgb, 255:red, 208; green, 2; blue, 27 }  ,draw opacity=1 ][dash pattern={on 4.5pt off 4.5pt}] (320,109.83) .. controls (320,99.06) and (328.73,90.33) .. (339.5,90.33) .. controls (350.27,90.33) and (359,99.06) .. (359,109.83) .. controls (359,120.6) and (350.27,129.33) .. (339.5,129.33) .. controls (328.73,129.33) and (320,120.6) .. (320,109.83) -- cycle ;
%Shape: Circle [id:dp2774477535030282] 
\draw  [color={rgb, 255:red, 208; green, 2; blue, 27 }  ,draw opacity=1 ][dash pattern={on 4.5pt off 4.5pt}] (389.33,110.5) .. controls (389.33,99.73) and (398.06,91) .. (408.83,91) .. controls (419.6,91) and (428.33,99.73) .. (428.33,110.5) .. controls (428.33,121.27) and (419.6,130) .. (408.83,130) .. controls (398.06,130) and (389.33,121.27) .. (389.33,110.5) -- cycle ;
%Straight Lines [id:da3307314581422639] 
\draw    (99.33,109.83) -- (71.67,109.83) ;
\draw [shift={(69.67,109.83)}, rotate = 360] [color={rgb, 255:red, 0; green, 0; blue, 0 }  ][line width=0.75]    (10.93,-3.29) .. controls (6.95,-1.4) and (3.31,-0.3) .. (0,0) .. controls (3.31,0.3) and (6.95,1.4) .. (10.93,3.29)   ;
%Straight Lines [id:da6434220361276249] 
\draw    (118.83,90.33) -- (119.62,54) ;
\draw [shift={(119.67,52)}, rotate = 91.25] [color={rgb, 255:red, 0; green, 0; blue, 0 }  ][line width=0.75]    (10.93,-3.29) .. controls (6.95,-1.4) and (3.31,-0.3) .. (0,0) .. controls (3.31,0.3) and (6.95,1.4) .. (10.93,3.29)   ;
%Straight Lines [id:da5426498456059765] 
\draw    (50.17,90.33) -- (117.92,52.97) ;
\draw [shift={(119.67,52)}, rotate = 151.12] [color={rgb, 255:red, 0; green, 0; blue, 0 }  ][line width=0.75]    (10.93,-3.29) .. controls (6.95,-1.4) and (3.31,-0.3) .. (0,0) .. controls (3.31,0.3) and (6.95,1.4) .. (10.93,3.29)   ;
%Straight Lines [id:da11521937528860582] 
\draw  [dash pattern={on 0.84pt off 2.51pt}]  (168.67,110.5) -- (140.33,109.88) ;
\draw [shift={(138.33,109.83)}, rotate = 1.26] [color={rgb, 255:red, 0; green, 0; blue, 0 }  ][line width=0.75]    (10.93,-3.29) .. controls (6.95,-1.4) and (3.31,-0.3) .. (0,0) .. controls (3.31,0.3) and (6.95,1.4) .. (10.93,3.29)   ;
%Curve Lines [id:da5132108161888977] 
\draw  [dash pattern={on 0.84pt off 2.51pt}]  (181.67,92) .. controls (155.53,97.88) and (130.68,81.99) .. (129.7,53.74) ;
\draw [shift={(129.67,52)}, rotate = 90] [color={rgb, 255:red, 0; green, 0; blue, 0 }  ][line width=0.75]    (10.93,-3.29) .. controls (6.95,-1.4) and (3.31,-0.3) .. (0,0) .. controls (3.31,0.3) and (6.95,1.4) .. (10.93,3.29)   ;
%Curve Lines [id:da4328336679740403] 
\draw    (141,51.33) .. controls (173.33,59.74) and (179.32,70.66) .. (187.41,89.25) ;
\draw [shift={(188.17,91)}, rotate = 246.63] [color={rgb, 255:red, 0; green, 0; blue, 0 }  ][line width=0.75]    (10.93,-3.29) .. controls (6.95,-1.4) and (3.31,-0.3) .. (0,0) .. controls (3.31,0.3) and (6.95,1.4) .. (10.93,3.29)   ;
%Straight Lines [id:da27616018439499446] 
\draw    (290.33,109.83) -- (318,109.83) ;
\draw [shift={(320,109.83)}, rotate = 180] [color={rgb, 255:red, 0; green, 0; blue, 0 }  ][line width=0.75]    (10.93,-3.29) .. controls (6.95,-1.4) and (3.31,-0.3) .. (0,0) .. controls (3.31,0.3) and (6.95,1.4) .. (10.93,3.29)   ;
%Straight Lines [id:da25653381448343215] 
\draw    (359,109.83) -- (387.33,110.46) ;
\draw [shift={(389.33,110.5)}, rotate = 181.26] [color={rgb, 255:red, 0; green, 0; blue, 0 }  ][line width=0.75]    (10.93,-3.29) .. controls (6.95,-1.4) and (3.31,-0.3) .. (0,0) .. controls (3.31,0.3) and (6.95,1.4) .. (10.93,3.29)   ;
%Straight Lines [id:da12754968055755334] 
\draw    (339.5,90.33) -- (339.66,53.33) ;
\draw [shift={(339.67,51.33)}, rotate = 90.24] [color={rgb, 255:red, 0; green, 0; blue, 0 }  ][line width=0.75]    (10.93,-3.29) .. controls (6.95,-1.4) and (3.31,-0.3) .. (0,0) .. controls (3.31,0.3) and (6.95,1.4) .. (10.93,3.29)   ;
%Straight Lines [id:da03661887976434808] 
\draw    (408.83,91) -- (341.4,52.33) ;
\draw [shift={(339.67,51.33)}, rotate = 29.83] [color={rgb, 255:red, 0; green, 0; blue, 0 }  ][line width=0.75]    (10.93,-3.29) .. controls (6.95,-1.4) and (3.31,-0.3) .. (0,0) .. controls (3.31,0.3) and (6.95,1.4) .. (10.93,3.29)   ;
%Curve Lines [id:da999169810899946] 
\draw    (285,96) .. controls (303.85,94.7) and (327.14,78.19) .. (328.91,53.88) ;
\draw [shift={(329,52)}, rotate = 91.51] [color={rgb, 255:red, 0; green, 0; blue, 0 }  ][line width=0.75]    (10.93,-3.29) .. controls (6.95,-1.4) and (3.31,-0.3) .. (0,0) .. controls (3.31,0.3) and (6.95,1.4) .. (10.93,3.29)   ;
%Curve Lines [id:da13036031110550161] 
\draw  [dash pattern={on 0.84pt off 2.51pt}]  (319.67,50.67) .. controls (298.22,51.32) and (278.67,66.54) .. (271.37,88.62) ;
\draw [shift={(270.83,90.33)}, rotate = 286.55] [color={rgb, 255:red, 0; green, 0; blue, 0 }  ][line width=0.75]    (10.93,-3.29) .. controls (6.95,-1.4) and (3.31,-0.3) .. (0,0) .. controls (3.31,0.3) and (6.95,1.4) .. (10.93,3.29)   ;
%Curve Lines [id:da39947380232878427] 
\draw [color={rgb, 255:red, 0; green, 0; blue, 0 }  ,draw opacity=1 ]   (151,31.33) .. controls (190.6,1.63) and (272.67,0.03) .. (307.95,34.28) ;
\draw [shift={(309,35.33)}, rotate = 225.55] [color={rgb, 255:red, 0; green, 0; blue, 0 }  ,draw opacity=1 ][line width=0.75]    (10.93,-3.29) .. controls (6.95,-1.4) and (3.31,-0.3) .. (0,0) .. controls (3.31,0.3) and (6.95,1.4) .. (10.93,3.29)   ;
%Curve Lines [id:da35303640561893623] 
\draw  [dash pattern={on 0.84pt off 2.51pt}]  (309.67,44.67) .. controls (260.17,76.35) and (208.71,74.05) .. (152.05,42.3) ;
\draw [shift={(150.33,41.33)}, rotate = 29.67] [color={rgb, 255:red, 0; green, 0; blue, 0 }  ][line width=0.75]    (10.93,-3.29) .. controls (6.95,-1.4) and (3.31,-0.3) .. (0,0) .. controls (3.31,0.3) and (6.95,1.4) .. (10.93,3.29)   ;

% Text Node
\draw (108,28) node [anchor=north west][inner sep=0.75pt]   [align=left] {NR};
% Text Node
\draw (113.67,101.67) node [anchor=north west][inner sep=0.75pt]   [align=left] {2};
% Text Node
\draw (45,102.33) node [anchor=north west][inner sep=0.75pt]   [align=left] {3};
% Text Node
\draw (265.17,102) node [anchor=north west][inner sep=0.75pt]   [align=left] {1};
% Text Node
\draw (326.5,27.4) node [anchor=north west][inner sep=0.75pt]    {$NR_{1}$};
% Text Node
\draw (332.5,101.9) node [anchor=north west][inner sep=0.75pt]    {$2_{1}$};
% Text Node
\draw (181.5,101.9) node [anchor=north west][inner sep=0.75pt]    {$1_{1}$};
% Text Node
\draw (402,102.4) node [anchor=north west][inner sep=0.75pt]    {$3_{1}$};

\end{tikzpicture}
\caption{State-transition graph of the ordered Zeckendorf Game with lower-order term tracking.}
\label{fig:state-machine-transitions}

\vspace{1em}
\small
    \begin{tabular}{r p{0.75\textwidth}}
        \textbf{Edges:} & Full lines indicate constant tuple length; dotted lines indicate length decreases by one. \\
        \textbf{Nodes:} & Labels follow Lemma \ref{lem:higher-repetitions}; green backgrounds indicate \textbf{pivot states}.\\
        \textbf{Subscripts:} & $xx_1$ denotes a lower-order tuple $(F_1, F_2, \dots, F_2)$, while $xx$ denotes a lower-order tuple $(F_2, \dots, F_2)$. 
    \end{tabular}
\end{figure}

    Based on the bipartite state graph shown in Figure \ref{fig:state-machine-transitions}, we define two possible pivot states for this configuration, determined by the \hit:
    \begin{itemize}
    \item $F_3$ is repeated (Pivot State 1), \label{case:pivot-1}
    \item No higher-index terms are repeated (Pivot State NR). \label{case:pivot-nr}
    \end{itemize}

    Notice that along each path from a pivot state back to itself or to another pivot state, the length of the tuple decreases by exactly one. For simplicity, we use a lower bound of $0$ moves for any transitions originating from states $2_1, 3_1, 2, \text{ and } 3$. Furthermore, because the initial tuple consists entirely of $F_2$ as its lower-index terms, if the game begins in state $2$ or $3$, we assume an automatic transition to state $NR$.

    Let us consider all possible paths between pivot states. 
    
    First, assume we begin at Pivot State 1. Table \ref{tab:pivot_state_1} illustrates the paths from Pivot State 1 to a state where the tuple has acquired a copy of $F_1$ on the far left.
    
    \begin{table}[h]
    \centering
    \caption{Sequence of moves to introduce $F_1$ on the left from Pivot State 1.}
    \label{tab:pivot_state_1}
    \begin{tabular}{|l|l|c|}
    \hline
    Path & State & Moves \\
    \hline
    Shared & $(F_2, \dots, F_2, F_3, F_3, \hit_{\ge4})$ (from Pivot State $1$) & 0 \\
    Sequence & $(F_2, \dots, F_2, F_1, F_4, \hit_{\ge4})$ & 1 \\
    \hline
    Path A & $(F_1, F_2, \dots, F_2, F_4, \hit_{\ge5})$ (to $NR_1$) & $\geq k-c_n-1$ \\
    \hline
    Path B & $(F_1, F_2, \dots, F_2, F_4, F_4, \hit_{\ge5})$ (to $2_1$) & $\geq k-c_n-1$\\
    & $(F_1, F_2, \dots, F_2, F_2, F_5, \hit_{\ge5})$ (to $3_1$ or $NR_1$) & $\geq k-c_n$\\
    & $(F_1, F_2, \dots, F_2, F_2, \hit)$ (to $NR_1$) & $\geq k-c_n$\\
    \hline
\end{tabular}
\end{table}

Following these moves, the tuple reaches state $NR_1$ after a minimum of $k-c_n-1$ moves. By appending an $F_1$ to the left end, the total number of lower-index terms becomes at least $k-c_n$. From state $NR_1$, the sequence must subsequently return to a pivot state. Table \ref{tab:NR_1} outlines these possible moves.

\begin{table}[h]
    \centering
    \caption{Sequence of moves from $NR_1$ to a Pivot State.}
    \label{tab:NR_1}
    \begin{tabular}{|l|l|c|}
    \hline
    Path & State & Moves \\
    \hline
    Shared & $(F_1, F_2, \dots, F_2, F_2, \hit)$ (from $NR_1$) & 0\\
    \hline
    Path A & $(F_1, \dots, F_2, F_1, F_3, \hit_{i\ge4})$ & 1 \\
    & $(F_1, F_1, \dots, F_2, F_3, \hit_{i\ge4})$ & $\geq k-c_n-2$ \\
    & $(F_2, \dots, F_2, F_3, \hit_{i\ge4})$ (to Pivot State NR) & $\geq k-c_n-1$ \\
    \hline
    Path B & $(F_1, \dots, F_2, F_1, F_3,F_3, \hit_{i\ge4})$ & $ 1$\\
    & $(F_1, F_1, \dots, F_3, F_3, \hit_{i\ge4})$ & $\geq k-c_n-2$ \\
    & $(F_2, \dots, F_2, F_3, F_3, \hit_{i\ge4})$ (to Pivot State 1) & $\geq k-c_n-1$ \\
    \hline
\end{tabular}
\end{table}

Consequently, transitioning from Pivot State 1 to another pivot state requires a minimum of $2k-2c_n -2$ moves. Completion of this path corresponds to a decrease of exactly 1 in the tuple's length.

Next, suppose we begin at Pivot State NR. Table \ref{tab:pivot_state_NR} displays the possible paths originating from this state. 

\begin{table}[h]
    \centering
    \caption{Sequence of moves from Pivot State NR to other states.}
    \label{tab:pivot_state_NR}
    \begin{tabular}{|l|l|c|}
    \hline
    Path & State & Moves \\
    \hline
    Shared & $(F_2, \dots, F_2, F_2, F_2, \hit)$ (from Pivot State NR) & 0 \\
    \hline
    Path A & $(F_2, \dots, F_2, F_1, F_3, \hit_{\ge4})$ & 1 \\
    & $(F_1, F_2, \dots, F_2, F_3, \hit_{\ge4})$ (to $NR_1$) & $\geq k-c_n-2$ \\
    \hline
    Path B & $(F_2, F_2, \dots, F_1, F_3, F_3, \hit_{\ge 4})$ & $1$ \\
    & $(F_1, F_2, \dots, F_3, F_3, \hit_{\ge4})$ (to $1_1$) & $\geq k-c_n-2$ \\
    \hline
\end{tabular}
\end{table}

If this sequence ends in state $NR_1$, the number of lower-index terms remains unchanged. Therefore, we can estimate the remaining moves required to reach the next pivot state by taking the lower bound from Table \ref{tab:NR_1} and subtracting 1.

Alternatively, if the sequence ends in state $1_1$, the required moves to return to Pivot State NR are detailed in Table \ref{tab:from_1_1}.

\begin{table}[h]
    \centering
    \caption{Sequence of moves from $1_1$ to Pivot State NR.}
    \label{tab:from_1_1}
    \begin{tabular}{|l|l|c|}
    \hline
    Path & State & Moves \\
    \hline
    Shared & $(F_1, F_2, \dots, F_2, F_3, F_3, \hit)$ (from $1_1$) & 0 \\
    \hline
    Path A & $(F_1, F_2, \dots, F_2, F_1, F_4, \hit_{\ge5})$ & 1 \\
    & $(F_1, F_1, F_2, \dots, F_2, F_4, \hit_{\ge5})$ & $\geq k-c_n-2$ \\
    & $(F_2, F_2, \dots, F_2, F_4, \hit_{\ge5})$ (to Pivot State NR) & $\geq k-c_n-1$ \\
    \hline
    Path B & $(F_1, F_2, \dots, F_2, F_1, F_4, F_4, \hit_{\ge5})$ & 1 \\
    & $(F_1, F_1, F_2, \dots, F_4, F_4, \hit_{\ge5})$ & $\geq k-c_n-2$ \\
    & $(F_2, F_2, \dots, F_4, F_4, \hit_{\ge5})$ (to $2$) & $\geq k-c_n-1$ \\
    & $(F_2, F_2, \dots, F_2, F_5, \hit_{\ge5})$ (to $3$ or Pivot State $NR$) & $\geq k-c_n-1$ \\
    & $(F_2, F_2, \dots, F_2, \hit)$ (to Pivot State $NR$) & $\geq k-c_n-1$ \\
    \hline
\end{tabular}
\end{table}

This demonstrates that the number of moves required to traverse from Pivot State NR through an intermediate state and back to a pivot state, thereby reducing the overall tuple length from $k$ to $k-1$, is bounded below by the minimum of the two possible path lengths:
\begin{enumerate}
    \item Path via $1_1$: $(k-c_n-2) + (k-c_n-1) = 2k-2c_n-3$
    \item Path via $NR_1$: $(k-c_n-2) + (k-c_n-2) = 2k-2c_n-4$
\end{enumerate}

Thus, the transition from a $k$-tuple to a $(k-1)$-tuple starting from the Pivot State $NR$ requires at least $2k-2c_n-4$ moves. 

Thus, to transition from a $k$-tuple to a $(k-1)$-tuple or, equivalently, to traverse between two Pivot States, we require at least $2k-2c_n-4$ moves. Summing these bounds yields the total estimate for $M(n)$:
\begin{align*}
M(n) &~\geq~ \frac{n^2}{4} + \sum_{k=c_n+5}^{\lfloor \frac{n}{2} \rfloor} (2k-2c_n-4) \\
&~=~ \frac{n^2}{4} + \lfloor \frac{n}{2} \rfloor \left( \lfloor \frac{n}{2} \rfloor + 1 \right) - (c_n+4)(c_n+5) - (2c_n+4) \left( \lfloor \frac{n}{2} \rfloor - c_n - 5 \right) \\
&~=~ \frac{n^2}{2} - nc_n + O(n) \\
&~=~ \frac{n^2}{2} - n\log_\phi(n) + o(n\log(n))
\end{align*}
as required.
\end{proof}

By analyzing a least decrement of the monovariant from switching and merging moves, we can now establish a refined upper bound with $n\log_\phi(n)$ term.

\begin{proof}[Proof of Theorem \ref{thm:refined-upper-bound-termination}]
    We shall think of $F_1$'s as an ordered set of stones that are rearranged in piles according to the rules of the game. For example, a pile $F_4$ consists of five $F_1$ stones. Note that at the start of the game, each $F_1$ forms a separate pile by itself. We say that a stone is being \emph{manipulated} if the pile in which it lies is involved in splitting or merging. 

    Recall the monovariant defined in the proof of Theorem \ref{thm:upper-bound-termination}: for a current decomposition $S = (F_{i_1}, \dots, F_{i_k})$, we let
\[
f(S) ~= ~\sum_{j=1}^k (k+1 - j) F_{i_j}.
\]
    
    We first show that each splitting and merging move decreases $f$ by at least $cN$, with $c$ a fixed constant and $N$ the number of stones being manipulated in the move. We also refer to a minimal decrement of $f$ for each type of move, which is given in the proof of Theorem \ref{thm:upper-bound-termination}.
    
    We consider four cases.
    \begin{itemize}
        \item When merging $F_{i-2}, F_{i-1}$, we have $N=F_i$, and $f$ decreases by at least $F_{i-2} \geq \frac{1}{4}F_i = \frac{1}{4}N$.
        \item When splitting $F_i, F_i$, we have $N = 2F_i$, and $f$ decreases by at least $F_{i-1} \geq \frac{1}{2} F_{i} = \frac{1}{4}N$.
        \item When splitting $F_2, F_2$, we have $N=4$, and $f$ decreases by $1 = \frac{1}{4} N$.
        \item When merging $F_1, F_1$, $f$ decreases by at least $1$. Noting that $N=2$, we have $1 \geq \frac{1}{4}N$.
    \end{itemize}
    Hence, taking $c=\frac{1}{4}$, our claim follows. So, to estimate the minimum amount that $f$ decreases due to merging and splitting moves, we can estimate the individual contribution of each stone. As we have shown that $f$ decreases by at least $c$ each time an individual stone is manipulated, then overall, $f$ decreases by at least 
    $$\sum_i c\cdot \#\text{of manipulations of $i$-th stone}.$$
    Let $F_l$ be the largest term in Zeckendorf decomposition of $n$. At the end of the game, a pile $F_l$ contains $F_l$ stones. Note that when some stone is being manipulated, the index of the pile in which it lies can increase by at most $2$: when merging $F_{i-2}, F_{i-1}$, the largest index increment of a stone is from a pile $F_{i-2}$ to a pile $F_i$;  when splitting $F_i, F_i$ or combining $F_1,F_1$, the largest index increment of a stone is from a pile $F_i$ to a pile $F_{i+1}$.

    Thus, each stone that ended the game in the pile $F_l$ was manipulated at least $\dfrac{l-1}{2}$ times. Note that $F_l \geq n/2$ and $l -1 \geq \log_\phi(n)/2$ for sufficiently large $n$. So, as a result of splitting and merging moves, $f$ decreases in the course of the game by at least 
    $$cF_l \dfrac{l-1}{2} ~\geq~ c \cdot \dfrac{n}{2} \cdot \dfrac{\log_\phi n}{4} ~=~ \dfrac{1}{32} n \log_\phi n.$$

    Thus, as we showed in the proof of Theorem \ref{thm:upper-bound-termination} that each switching move decreases $f$ by at least $1$, there can be no more than 
    $$\dfrac{n(n-1)}{2} - \dfrac{1}{32}n\log_\phi n$$ switching moves. The total number of splitting and merging moves cannot exceed $3n+1$ by \cite[Theorem 1.3]{li2020deterministic}. Hence, the total number of moves is at most:
    $$\dfrac{n(n-1)}{2} - \dfrac{1}{32}n\log_\phi n +3n + 1~=~ \dfrac{n^2}{2} - \dfrac{1}{32}n\log_\phi n + o(n\log_\phi n),$$
    as required.
\end{proof}

\begin{remark}
The longest-game strategy described in Conjecture~\ref{Longest strategy} is supported by empirical simulations. However, a rigorous proof establishing its optimality remains an open problem.
\end{remark}

We also remark that in the proof above, we can improve a constant before $n\log_\phi(n)$ term up to $1/(8\phi)\approx 0.077$ by noting that, asymptotically, $F_l \geq n/\phi$ and $l \geq \log_\phi(n)$. However, assuming that LGS algorithm indeed gives a maximal game length, numerical simulations suggest that the constant before the $n\log_\phi(n)$ term is at least 0.4 and should be much lower that 1, as demonstrated on Figure~\ref{fig:error_term_estimation}. Hence, while Theorem \ref{thm:lower-bound-termination} and Theorem \ref{thm:refined-upper-bound-termination} allow us to establish asymptotic behaviour of $\dfrac{n^2}{2} - M(n)$, they contain too crude estimates to give an insight for the explicit constant before $n\log_\phi(n)$.

We now establish a sharp lower bound for minimum game length, remarking that this is equivalent to the minimum game length of the classical Zeckendorf Game and is achieved via a game where no switch moves occur.

\begin{proof}[Proof of Proposition \ref{Shortest Game}]
To prove the following proposition, we use the total number of summands as a monovariant, which decreases with each move. Note that switch and split moves preserve the total number of summands, while a merge reduces it by one. Since the game starts with \(n\) summands (all \(F_1\)) and ends with \(Z(n)\) summands (the number of terms in the Zeckendorf decomposition of \(n\)), the minimum number of moves required to reach the terminal state is \(n - Z(n)\).

This bound is achieved by applying the greedy algorithm for constructing the Zeckendorf decomposition with the fewest possible moves. Starting from the initial configuration of \(n\) copies of \(F_1\), we repeatedly merge adjacent Fibonacci numbers from right to left to create the largest possible Fibonacci term at each step. Once the largest term is formed, we leave it in place and recursively apply the same strategy to the remaining entries to its left. Each merge reduces the total number of summands by one, and the process continues until the configuration matches the Zeckendorf decomposition of \(n\).

\end{proof}

It is worth highlighting an intriguing phenomenon that emerges when analyzing randomly played games, in which each legal move is chosen with equal probability. The empirical evidence gathered from extensive simulations motivates the following conjecture.

\begin{conjecture}\label{conj:log-gaussian}
As \( n \to \infty \), the distribution of the number of moves in a randomly played game, where all legal moves are equally likely, converges to a log-normal distribution. Under these conditions, the game dynamics are sufficiently symmetric such that each player is equally likely to win.
\end{conjecture}

This conjecture is supported by simulation data presented in Figure~\ref{fig:1_random_distribution} and Figure~\ref{fig:2_random_outcomes}, which suggest that the observed move counts exhibit the heavy-tailed, right-skewed characteristics consistent with a log-normal distribution. Furthermore, the symmetry in the outcome frequencies aligns with the theoretical expectation of equal win probabilities for both players under a uniformly random move selection process.

%%%%%%%%%%%%%%%%%%%%%%%%%%%%%%%%%%%%%%%%%%%%%%%%%%%%%%%%%%%%%
%%%%%%%%%%%%%%%%%%%%%%%%%%%%%%%%%%%%%%%%%%%%%%%%%%%%%%%%%%%%%

\section{Game Winning Strategy}

Through exhaustive simulations, we investigated the existence of deterministic winning strategies for both players in the Two-Player Ordered Zeckendorf Game. Our computational experiments reveal that Player 1 possesses a winning strategy for all initial values up to \( n = 17 \). At \( n = 18 \), we observe the first instance where Player 2 can force a win under optimal play, indicating a potential transition point in the game’s strategic landscape. Interestingly, in all simulated games for \( n \leq 25 \), Player 2 is only able to force a win once, suggesting that Player 1 generally has a significant advantage, especially for smaller values of \( n \).

Due to the combinatorial explosion in the number of game states and legal move sequences, a full resolution of the winner for \( n > 25 \) was computationally infeasible. The structure of the game tree grows rapidly with \( n \), and optimal strategies require deep lookahead and pruning heuristics that go beyond naive enumeration.

The existence of a general winning strategy for either player remains an open problem. Furthermore, due to the game's nontrivial dynamics and the subtle impact of move orderings on future legal moves, it is possible that the game exhibits regions of alternating player advantage or chaotic behavior in terms of outcome predictability.

Future work could include developing a minimax-based heuristic search with alpha-beta pruning to more efficiently explore deeper game trees and extract structural insights into what governs winning positions. Additionally, identifying game invariants or monotonic features that correlate with winning states could pave the way toward a formal characterization of optimal strategies.

%%%%%%%%%%%%%%%%%%%%%%%%%%%%%%%%%%%%%%%%%%%%%%%%%%%%%%%%%%%%%
%%%%%%%%%%%%%%%%%%%%%%%%%%%%%%%%%%%%%%%%%%%%%%%%%%%%%%%%%%%%%

\section{Conclusion}

We have introduced and analyzed a novel ordered variant of the Zeckendorf game. We established key theoretical results, including termination guarantees, asymptotic bounds on the maximum number of moves, and lower bounds on the minimal number of required moves. Our study combines combinatorial structure with game dynamics, revealing rich mathematical behavior in even small initial configurations.

Despite our progress, many open questions remain.

\begin{itemize}
    \item Optimal Strategy Characterization: Can we rigorously prove the existence and form of an optimal winning strategy for Player 1 or Player 2 across all \( n \)? Is the observed dominance of Player 1 for small \( n \) a universal phenomenon, or do other exceptions arise at larger scales?

    \item Maximal Game Length Proof: While we conjecture a quadratic asymptotic growth and provide a priority-based strategy achieving near-maximal length, a formal proof of optimality remains open. Can analytic or combinatorial tools be developed to close this gap?

    \item Generalizations to Other Recurrences: How does the introduction of order and switching moves affect games based on generalized Fibonacci sequences, such as \( k \)-bonacci sequences or non-constant recurrences? Do analogous termination and length results hold?

    \item Probabilistic Analysis of Random Play: Our empirical evidence suggests a log-normal distribution of move counts under uniform random play. Can this distribution be derived rigorously? What are the deeper probabilistic and ergodic properties of the random dynamics?

    \item Computational Complexity and Algorithmic Approaches: Given the rapid growth of the game tree, can efficient algorithms or heuristics be developed to determine winning positions for larger \( n \)? Are there polynomial-time methods to approximate optimal strategies or game lengths?

    \item Structural Invariants and Monovariants: Are there additional invariants or monotonic quantities beyond those currently identified that govern the evolution of game states? How might these guide strategic reasoning or simplify analysis?

    \item Multiplayer or Alliance Variants: Extending to more than two players or cooperative alliances—how do order and adjacency constraints reshape the strategic landscape? Can these variants be fully characterized or related to known combinatorial game theory frameworks?

\end{itemize}

These questions point toward deeper connections with additive number theory, algorithmic game theory, and dynamical systems. We believe this variant offers a promising platform for further exploration of combinatorial games and integer representations.

All simulations and visualizations used in our analysis were implemented in Python and C++, and the full codebase is publicly available at \url{https://github.com/vraneskoy/zeckendorf_ordered_game}.

%%%%%%%%%%%%%%%%%%%%%%%%%%%%%%%%%%%%%%%%%%%%%%%%%%%%%%%%
%%%%%%%%%%%%%%%%%%%%%%%%%%%%%%%%%%%%%%%%%%%%%%%%%%%%%%%

\section{Appendix}
The following two figures provide computational evidence in support of Conjecture \ref{conj:log-gaussian}.
\begin{figure}[H]
    \centering
    \includegraphics[width=0.9\linewidth]{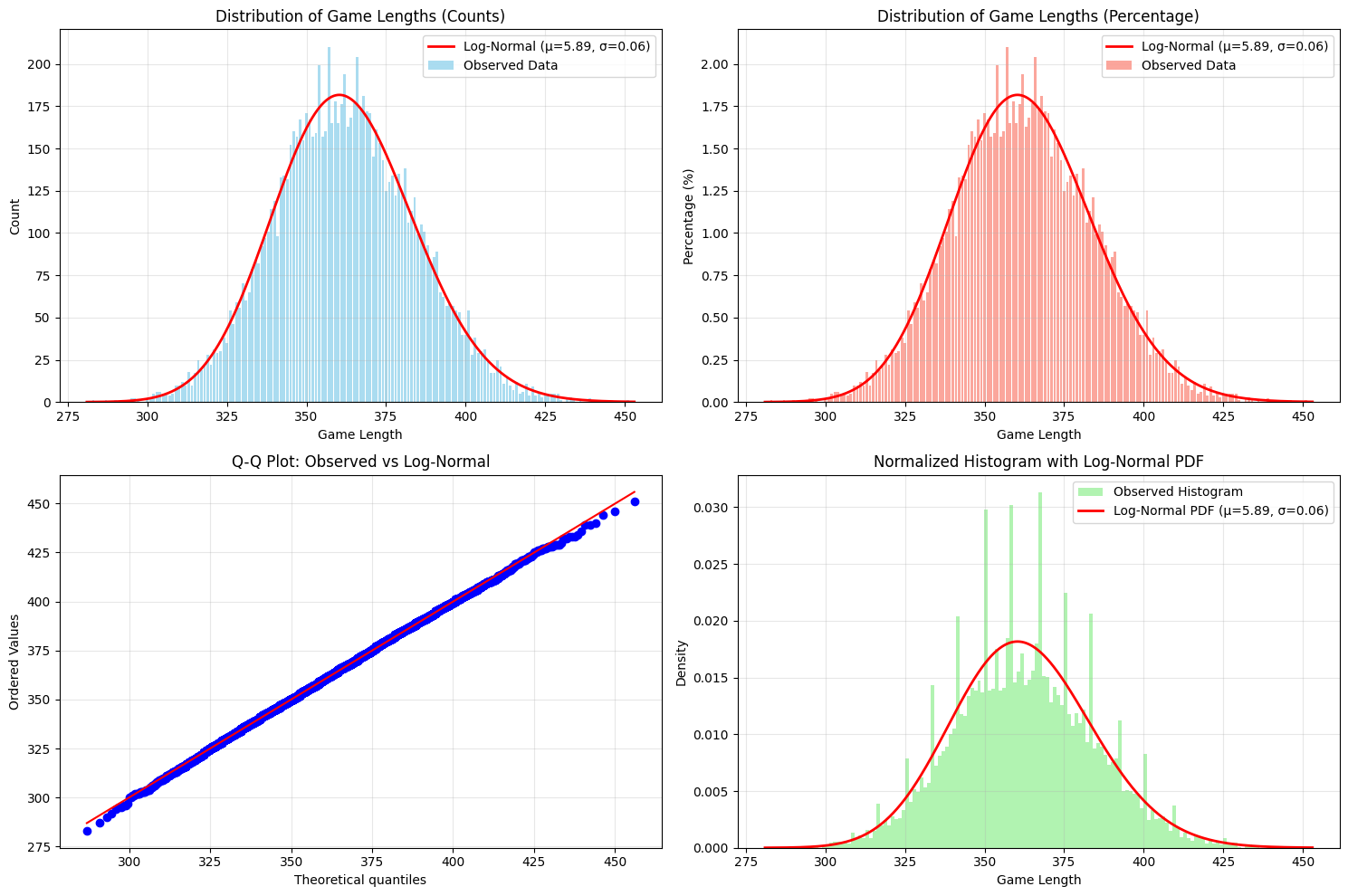}
    \caption{Frequency graphs of the number of moves in 10,000 simulations of the Zeckendorf Game with random moves when $n = 150$ with
the best fit log Gaussian over the data points.}
    \label{fig:1_random_distribution}
\end{figure}

\begin{figure}[H]
    \centering
    \includegraphics[width=0.35\linewidth]{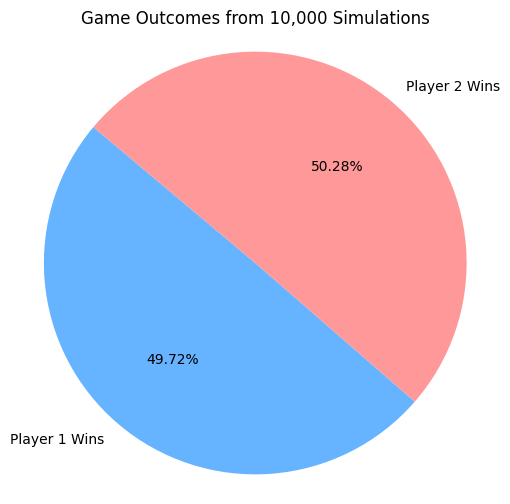}
    \caption{The distribution of outcomes for 10,000 random games for $n = 150$.}
    \label{fig:2_random_outcomes}
\end{figure}

Corollary \ref{cor:bound-termination} gives us an error term in our upper bound on maximal game length of
$n^2/2 - M(n) = \Theta(n \log_\phi(n))$. The following figure provides an estimation for the constant before the $n \log_\phi(n)$ term.

\begin{figure}[H]
    \centering
    \includegraphics[width=0.9\linewidth]{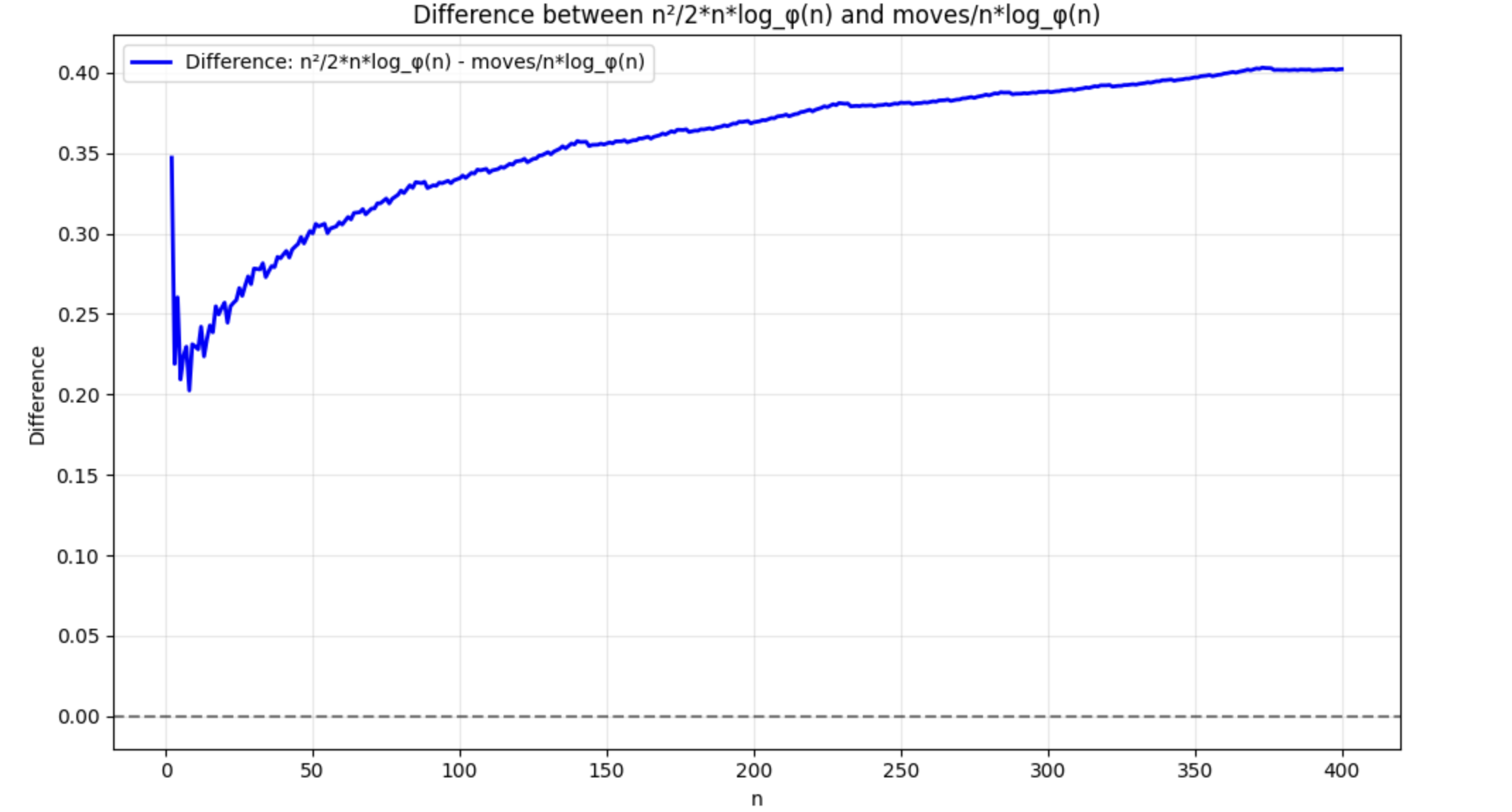}
    \caption{As \( n \to \infty \), the ratio $\dfrac{\frac{n^2}{2}-M(n)}{n\log_\phi(n)}$} appears to converge to a number bigger than 0.4 and much lower than 1.
    \label{fig:error_term_estimation}
\end{figure}

\bibliographystyle{plain}

\bibliography{mybib}

\ \\

\end{document}